\documentclass[]{amsart}
\usepackage[T1]{fontenc}
\usepackage{amsthm, amssymb, amsmath}
\usepackage{mathrsfs, latexsym, color}
\usepackage[pdftex]{}
\numberwithin{equation}{section}

\theoremstyle{plain}% default
\newtheorem{theorem}{Theorem}[section]
\newtheorem{proposition}[theorem]{Proposition}
\newtheorem{lemma}[theorem]{Lemma}
\newtheorem{corollary}[theorem]{Corollary}

\theoremstyle{definition}

\theoremstyle{remark}

\renewcommand{\bar}{\overline}

\newcommand{\abs}[1]{\lvert#1\rvert}

\newcommand{\R}{\mathbb{R}}
\newcommand{\N}{\mathbb{N}}
\newcommand{\Z}{\mathbb{Z}}
\newcommand{\Q}{\mathbb{Q}}
\newcommand{\C}{\mathbb{C}}

\usepackage[all]{xy}
\usepackage{pgf}
\usepackage{tikz}
\usetikzlibrary{arrows,decorations.markings,shapes}
\tikzstyle{vertex}=[circle, draw, fill=black, inner sep=0pt, minimum width=6pt]
\tikzstyle{vert}=[circle, draw=black, fill=white, inner sep=0pt, minimum width=6pt]
\tikzstyle{zc}=[circle, draw, fill=black, inner sep=0pt, minimum width=6pt]
\tikzstyle{pc}=[circle, fill=white, draw=black, inner sep=1pt, minimum width=10pt, font=\tiny] %positive charge
\tikzstyle{nc}=[circle, draw=black, inner sep=1pt, minimum width=10pt, font=\tiny] % style=dashed,
\tikzstyle{pedge}=[draw,-]
\tikzstyle{wedge}=[draw,-,postaction={decorate}, decoration={markings,mark = at position 0.55 with {\arrow{stealth} } }]
\tikzstyle{dpedge}=[draw,-,postaction={decorate}]
\tikzstyle{wwedge}=[draw,-,postaction={decorate}, decoration={markings,mark = at position 0.5 with {\arrow{stealth} }, mark = at position 0.6 with {\arrow{stealth} } }]
\tikzstyle{wwedge2}=[draw,-,postaction={decorate}, decoration={markings,mark = at position 0.45 with {\arrow{stealth} }, mark = at position 0.65 with {\arrow{stealth} } }]
\tikzstyle{wwwedge}=[draw,-,postaction={decorate}, decoration={markings,mark = at position 0.65 with {\arrow{stealth} },mark = at position 0.45 with {\arrow{stealth} }, mark = at position 0.55 with {\arrow{stealth} } }]
\tikzstyle{wwwedge2}=[draw,-,postaction={decorate}, decoration={markings,mark = at position 0.7 with {\arrow{stealth} },mark = at position 0.4 with {\arrow{stealth} }, mark = at position 0.55 with {\arrow{stealth} } }]
\tikzstyle{wwwwedge}=[draw,-,postaction={decorate}, decoration={markings,mark = at position 0.7 with {\arrow{stealth} },mark = at position 0.4 with {\arrow{stealth} }, mark = at position 0.5 with {\arrow{stealth} }, mark = at position 0.6 with {\arrow{stealth} } }]
\tikzstyle{wwwnedge}=[draw,densely dashed,postaction={decorate}, decoration={markings,mark = at position 0.65 with {\arrow{stealth} },mark = at position 0.45 with {\arrow{stealth} }, mark = at position 0.55 with {\arrow{stealth} } }]
\tikzstyle{wwnedge}=[draw,densely dashed,postaction={decorate}, decoration={markings,mark = at position 0.5 with {\arrow{stealth} }, mark = at position 0.6 with {\arrow{stealth} } }]
\tikzstyle{wnedge}=[draw,densely dashed,postaction={decorate}, decoration={markings,mark = at position 0.55 with {\arrow{stealth} } }]
\tikzstyle{wnedge2}=[draw,densely dashed,postaction={decorate}, decoration={markings,mark = at position 0.6 with {\arrow{stealth} } }]
\tikzstyle{dnedge}=[draw,densely dashed,postaction={decorate}]
\tikzstyle{nedge}=[draw,densely dashed]
\tikzstyle{weight2}= [draw=white, fill=white, font=\scriptsize]
\tikzstyle{weight}= [font=\scriptsize]
\tikzstyle{empty}=[circle, draw=white, inner sep=2pt, fill=white, minimum width=4pt]
\tikzstyle{ghost}=[circle, draw=black, inner sep=1pt, style=densely dashed, minimum width=6pt, font=\tiny]
\tikzstyle{ghostc}=[circle, draw=black, inner sep=1pt, style=densely dashed, minimum width=10pt, font=\tiny]
\tikzstyle{dedge}=[draw,very thick,dotted]
\usepackage{multirow}
\renewcommand{\bar}{\overline}

\title[Small-span Hermitian matrices over quadratic integer rings]{Small-span Hermitian matrices over \\ quadratic integer rings}
\author{ Gary Greaves }
\address{Research Center for Pure and Applied Mathematics, Graduate School of Information Sciences, Tohoku University, Sendai, Japan.}
\email{grwgrvs@ims.is.tohoku.ac.jp}
\date{}

\subjclass[2010]{Primary 11C08, 15B36. Secondary 05C22, 05C50, 11C20, 15B33}

\begin{document}
	
	\begin{abstract}
		A totally-real polynomial in $\Z[x]$ with zeros $\alpha_1 \leqslant \alpha_2 \leqslant \dots \leqslant \alpha_d$ has span $\alpha_d - \alpha_1$.
		Building on the classification of all characteristic polynomials of integer symmetric matrices having small span (span less than $4$), we obtain a classification of small-span polynomials that are the characteristic polynomial of a Hermitian matrix over some quadratic integer ring.
		Taking quadratic integer rings as our base, we obtain as characteristic polynomials some low-degree small-span polynomials that are not the characteristic (or minimal) polynomial of any integer symmetric matrix.
	\end{abstract}

\maketitle

\section{Introduction and preliminaries}

\subsection{Introduction}

Let $f$ be a totally-real monic polynomial of degree $d$ having integer coefficients with zeros $\alpha_1 \leqslant \alpha_2 \leqslant \dots \leqslant \alpha_d$.
The \textbf{span} of $f$ is defined to be $\alpha_d - \alpha_1$.
In this article, for Hermitian matrices $A$ over quadratic integer rings, we are interested in the characteristic polynomials $\chi_A(x) := \det(xI-A)$ and the minimal polynomials $p_A(x)$.
We build on the work of McKee~\cite{McKee:SmallSpan10} who classified all integer symmetric matrices whose characteristic polynomials have span less than $4$.
Following McKee, we call a totally-real monic integer polynomial \textbf{small-span} if its span is less than $4$ and a Hermitian matrix is called small-span if its characteristic polynomial is small-span.

A set consisting of an algebraic integer and all of its (Galois) conjugates is called a \textbf{conjugate set}.
Given a real interval $I$ of length $l$, it is an interesting question as to how many conjugate sets are contained in $I$.
For $l > 4$ it is known that $I$ contains infinitely many conjugate sets \cite{Rob62:largeSpan}.
For $l = 4$, with $I = [a,a+4]$, for some $a \in \Z$, there are infinitely many conjugate sets contained in $I$ \cite{Kron:cyclo57} -- these conjugate sets correspond to cosine polynomials and their translates (see Section~\ref{sec:4}).
For $a \not \in \Z$ the question is still open.
On the other hand, if $l < 4$ then there are only finitely many conjugate sets contained in $I$ \cite{Schur:FiniteSmallSpan}.
Therefore, up to equivalence (described below), there are only finitely many irreducible small-span polynomials of any given degree.
A list of small-span polynomials up to degree $8$ was produced by Robinson~\cite{Rob:SmallSpan64}, (this list has been recently extended up to degree $14$ \cite{Cap:SmallSpan10} and even more recently up to degree $15$ \cite{Flamm11:smallspan}). 

Let $K$ be a quadratic extension of $\Q$ and let $\mathcal O_K$ be its ring of integers.
We classify all small-span Hermitian matrices over $\mathcal O_K$ for all quadratic number fields $K$.
In doing so, we obtain, as characteristic polynomials of symmetric $\mathcal O_K$-matrices, small-span polynomials that are not the minimal polynomial of any Hermitian $\Z$-matrix.
But, as we shall see, there still remain some that are not even minimal polynomials of any Hermitian $\mathcal O_K$-matrix.

Let $R \subset \C$ be a (unitary) ring and let $A$ be a Hermitian matrix over $R$.
Both $p_A(x)$ and $\chi_A(x)$ are monic polynomials in $R[x]$.
Since $A$ is Hermitian, $p_A(x)$ is separable and both $p_A(x)$ and $\chi_A(x)$ are totally real.
Let $\mathfrak T$ be the set of all totally-real separable monic integer polynomials and let $H(R)$ be the set of all Hermitian matrices over $R$.
Define
\[
	\mathfrak S(R) := \left \{ f(x) \in \mathfrak T \; | \; p_M(x) = f(x),\; \exists M \in H(R) \right \},
\]
and
\[
	\mathfrak d(R) := \min \left \{ \deg f(x) \; |\; f(x) \in \mathfrak T \backslash \mathfrak S(R) \right \}.
\]
Of course $\mathfrak d(R)$ may not exist for some rings $R$.
For example, when $R = \R$ we write $\mathfrak d(R) = \infty$.
Clearly, for $R_1 \subset R_2 \subset \C$, we have $\mathfrak d(R_1) \leqslant \mathfrak d(R_2)$.

Estes and Guralnick~\cite{EsGu:MinPolISM} were interested in the set $\mathfrak S(\Z)$.
They proved that $\mathfrak S(\Z)$ contains all degree-$k$ totally-real separable monic integer polynomials for $k \leqslant 4$, giving $\mathfrak d(\Z) > 4$.
In fact, what they proved is even stronger: they showed that for any monic separable totally-real integer polynomial $f$ of degree $n$ at most $4$, there exists a $2n \times 2n$ integer symmetric matrix having $f$ as its minimal polynomial.

It is well established \cite{Estes:Eigen92} that, for every totally-real algebraic integer $\alpha$, there exists an integer symmetric matrix $A$ having $\alpha$ as an eigenvalue, and hence the minimal polynomial of $\alpha$ divides $p_A$.
Hoffman~\cite{Hoff:EigenOGraph75,Hoffman:1972fk} showed that this result implied the same result for the adjacency matrices of undirected graphs without loops or multiple edges.
In their paper concerned with the set $\mathfrak S(\Z)$, Estes and Guralnick~\cite[page 84]{EsGu:MinPolISM} made the conjecture that any monic separable totally-real integer polynomial is the minimal polynomial of some integer symmetric matrix.
This conjecture was shown to be false by Dobrowolski~\cite{Dob:Int08} who showed that if a degree $n$ polynomial $f$ is the minimal polynomial of an integer symmetric matrix then the absolute value of its discriminant is at least $n^n$.
The falsity of the conjecture was deduced by producing an infinite family of totally-real monic separable integer polynomials violating this bound.
The lowest degree of any of these polynomials is $2880$, giving $5 \leqslant \mathfrak d(\Z) \leqslant 2880$.
It was shown in \cite{McKee:SmallSpan10} that there exist polynomials of low degree that are not the minimal polynomial of any integer symmetric matrix, the lowest degree being $6$.
Therefore $\mathfrak d(\Z) \in \{5,6\}$.
This is a consequence of the classification of small-span $\Z$-matrices; on Robinson's list \cite{Rob:SmallSpan64} of small-span polynomials, there are three polynomials of degree $6$ that are not the minimal polynomials of any small-span integer symmetric matrix and hence of any integer symmetric matrix.
It is still unknown whether $\mathfrak d(\Z) = 5$ or $6$.

In this paper we characterise every small-span Hermitian matrix over a quadratic integer ring.
In Section~\ref{sec:2} we describe our computations and the classification is proved in Section~\ref{sec:3}.
In Section~\ref{sec:4}, using a quadratic integer ring as our base, we find as characteristic polynomials some degree-$6$ small-span polynomials that are not the characteristic (or minimal) polynomial of any integer symmetric matrix.
We also exhibit a degree-$6$ polynomial that is not the characteristic or minimal polynomial of any Hermitian matrix over any quadratic integer ring.
Hence we obtain the following theorem
\begin{theorem}\label{thm:main}
	Let $K$ be a quadratic number field.
	Then $\mathfrak d(\mathcal O_K) \in \{5,6\}$. 
\end{theorem}

\subsection{Equivalence and graphical interpretation}

Let $S$ be a subset of $\C$.
We find it convenient to view Hermitian $S$-matrices as the adjacency matrices of weighted directed graphs $G$ with edge-weights $w$.
So we define these as the pair $(G,w)$ where $w$ maps pairs of (not necessarily distinct) vertices of $G$ to $S$ and satisfies $w(u,v) = \bar{w(v,u)}$.
The adjacency matrix of $(G,w)$ is given by $A = (w(u,v))$.
We call $G = (G,w)$ an \textbf{$S$-graph}.
This way of viewing Hermitian matrices has been used before in \cite{McKee:IntSymCyc07,McKee:SmallSpan10, GTay:cyclos10, Greaves:CycloEG11, Greaves:CycloRQ11}.
By a \textbf{subgraph} $H$ of $G$ we mean an induced subgraph: a subgraph obtained by deleting vertices and their incident edges.
Taking subgraphs corresponds to taking principal submatrices of the adjacency matrix.
We also say that $G$ \textbf{contains} $H$ and that $G$ is a \textbf{supergraph} of $H$.
The notions of a cycle/path/triangle etc.\ carry through in an obvious way from those of an undirected unweighted graph.
We call a vertex $v$ \textbf{charged} with charge $w(v,v)$ if $w(v,v) \ne 0$, otherwise $v$ is called \textbf{uncharged}.
A graph is called \textbf{charged} if it contains at least one charged vertex, and \textbf{uncharged} otherwise.
We will interchangeably speak of both graphs and their adjacency matrices.

% \section{Equivalence}

Now we describe equivalence.
We write $M_n(R)$ for the ring of $n \times n$ matrices over a ring $R \subseteq \C$. 
Let $U_n(R)$ denote the unitary group of matrices $Q$ in $M_n(R)$.
These satisfy $QQ^* = Q^*Q = I$, where $Q^*$ denotes the Hermitian transpose of $Q$.
Conjugation of a matrix $M \in M_n(R)$ by a matrix in $U_n(R)$ preserves the eigenvalues of $M$ and the base ring $R$.
Now, $U_n(R)$ is generated by permutation matrices and diagonal matrices of the form
\[
	\operatorname{diag}(1,\dots,1,u,1,\dots,1),
\]
where $u$ is a unit in $R$.
Let $D$ be such a diagonal matrix having $u$ in the $j$-th position.
Conjugation by $D$ is called a $u$-\textbf{switching} at vertex $j$.
This has the effect of multiplying all the out-neighbour edge-weights $w(j,l)$ of $j$ by $u$ and all the in-neighbour edge-weights $w(l,j)$ of $j$ by $\bar u$.
The effect of conjugation by permutation matrices is just a relabelling of the vertices of the corresponding graph.
Let $L$ be the Galois closure of the field generated by the elements of $R$ over $\Q$.
Let $A$ and $B$ be two matrices in $M_n(R)$.
We say that $A$ is \textbf{strongly equivalent} to $B$ if $A = \sigma(QBQ^*)$ for some $Q \in U_n(R)$ and some $\sigma \in \operatorname{Gal}(L/\Q)$, where $\sigma$ is applied componentwise to $QBQ^*$.
Let $f$ and $g$ be totally-real monic integer polynomials of degree $d$ having zeros $\alpha_j$ and $\beta_j$ respectively. 
We consider $f$ and $g$ to be \textbf{equivalent} if for some $c \in \Z$ and $\varepsilon = \pm 1$ each $\alpha_j = \varepsilon \beta_j + c$.
It is clear that the span is preserved under this equivalence.
If a totally-real monic integer polynomial has span less than $4$ then it is equivalent to a monic integer polynomial whose zeros are all contained inside the interval $[-2,2.5)$.
By setting $\varepsilon = 1$ and $c = \lfloor \alpha_d \rfloor - 2$, one can see that each small-span polynomial is equivalent to a monic integer polynomial whose zeros are contained inside the interval $[-2,3)$.
Moreover, suppose that $f$ is a small-span polynomial with $2.5 \leqslant \alpha_d < 3$.
Setting $\varepsilon = -1$ and $c = 1$, one can see that, in fact, $f$ and hence each small-span polynomial, is equivalent to a monic integer polynomial whose zeros are contained inside the interval $[-2,2.5)$.
The matrices $A$ and $B$ are called \textbf{equivalent} if $A$ is strongly equivalent to $\pm B + cI$ for some $c \in \Z$.
Observe that two small-span $\mathcal O_K$-matrices are equivalent precisely when their characteristic polynomials are equivalent. 
The notions of equivalence and strong equivalence carry through to graphs in the natural way and, since all possible labellings of a graph are strongly equivalent, we do not need to label the vertices.

When we say ``$G$ is a graph,'' we mean that $G$ is some $S$-graph where $S$ is some subset of $\C$ and for our notion of equivalence, we take $R$ to be the ring generated by the elements of $S$.

\section{Computation of small-span matrices up to size $8 \times 8$}
\label{sec:2}

In this section we describe our computations and deduce some restrictions to make the computations feasible.

\subsection{Interlacing and row restrictions}
A Hermitian matrix $A$ is called \textbf{decomposable} if there exists a permutation matrix $P$ such that $PAP^\top$ is a block diagonal matrix of more than one block; otherwise $A$ is called \textbf{indecomposable}.
In the classification of small-span Hermitian matrices, it suffices to consider only indecomposable matrices since the set of eigenvalues of a decomposable matrix is just the collection of the eigenvalues of each of its blocks.

The following theorem of Cauchy~\cite{Cau:Interlace,Fisk:Interlace05,Hwang:Interlace04} plays a central r\^{o}le in this paper.

\begin{lemma}[Interlacing Theorem] \label{lem:interlacing} Let $A$ be an $n~\times~n$ Hermitian matrix with eigenvalues $\lambda_1 \leqslant \dots \leqslant \lambda_n$. Let $B$ be an $(n - 1)~\times~(n - 1)$ principal submatrix of $A$ with eigenvalues $\mu_1 \leqslant \dots \leqslant \mu_{n-1}$. Then the eigenvalues of $A$ and $B$ interlace. Namely,
	\[
		\lambda_1 \leqslant \mu_1 \leqslant \lambda_2 \leqslant \mu_2 \leqslant \dots \leqslant \mu_{n-1} \leqslant \lambda_n.
	\]
\end{lemma}

\begin{corollary}\label{cor:interlaceSpan}
	Let $A$ be an $n \times n$ Hermitian matrix with $n \geqslant 2$, and let $B$ be an $(n-1) \times (n-1)$ principal submatrix. Then the span of $A$ is at least as large as the span of $B$. Moreover, if $A$ has all its eigenvalues in the interval $[-2,2.5)$, then so does $B$. 
\end{corollary}

In view of this corollary, given a matrix that contains a matrix that it not equivalent to a small-span matrix having all of its eigenvalues in the interval $[-2,2.5)$, we can instantly disregard it since it is not a small-span matrix.
In fact, by Corollary~\ref{cor:interlaceSpan}, we can iteratively compute small-span matrices as described below.

Using interlacing (Lemma~\ref{lem:interlacing}) we can rapidly restrict the possible entries for the matrices that are of interest to us. This is just a trivial modification of a lemma~\cite[Lemma 4]{McKee:SmallSpan10} used in the classification of small-span integer symmetric matrices.

\begin{lemma}\label{lem:restrictEntries}
	Let $A$ be a small-span Hermitian matrix. Then all entries of $A$ have absolute value less than $2.5$, and all off-diagonal entries have absolute value less than $2$.
\end{lemma}
\begin{proof}
	Let $a$ be a diagonal entry in $A$. 
	Then since $(a)$ has $a$ as an eigenvalue, interlacing shows that $A$ has an eigenvalue with modulus at least $|a|$. 
	Our restriction on the eigenvalues of $A$ shows that $|a| < 2.5$.
	
	Let $b$ be an off-diagonal entry of $A$.
	Then deleting the other rows and columns gives a submatrix of the shape
	$$\begin{pmatrix}
		a & b \\ \bar{b} & c
	\end{pmatrix}.$$
	By repeated use of Corollary~\ref{cor:interlaceSpan}, this submatrix must have span less than $4$, giving $\sqrt{(a-c)^2+4\abs{b}^2} < 4$.
	This implies $|b|< 2$.
\end{proof}

\subsection{Small-span matrices over quadratic integer rings}

In order to obtain our results, we use a certain amount of computation.
By Corollary~\ref{cor:interlaceSpan}, if $A$ is a small-span matrix then so is any principal submatrix of $A$.
Hence, to compute small-span matrices we can start by creating a list of all $1 \times 1$ small-span matrices, then for each matrix in the list we consider all of its supergraphs and add to our list any that are small-span.
The list is then pruned with the goal of having at most one representative for each equivalence class.
Since there is no canonical form in our equivalence class, we can only prune our list to some limited extent.
We can repeat this growing process until all small-span matrices of the desired size are obtained.
The algorithm we use is essentially the same as the one described in \cite{McKee:SmallSpan10}, with modifications to deal with irrational elements.
Up to equivalence, for $d \in \Z$, we compute all small-span $\mathcal O_{\Q(\sqrt{d})}$-matrices up to size $8 \times 8$, and in doing so, we also compute Hermitian matrices whose eigenvalues satisfy the small-span condition but whose characteristic polynomials do not have integer coefficients.
As it turns out, these do not cause a problem: we will see that for $n > 6$, an $n \times n$ matrix whose eigenvalues satisfy the small-span condition also has an integer characteristic polynomial.

To make our computation more efficient, we can bound the number of nonzero entries in a row of a small-span matrix.
The amount of computation required to prove this lemma varies according to the ring over which we are working.
The ring $\mathcal O_{\Q(\sqrt{-3})} = \Z[\omega]$ (where we take $\omega = 1/2 + \sqrt{-3}/2$) requires the most work.

\begin{lemma}\label{lem:boundRowEntries}
Let $A$ be a small-span Hermitian $\Z[\omega]$-matrix with all eigenvalues in the interval $[-2,2.5)$.
Then each row of $A$ has at most $4$ nonzero entries.
\end{lemma}
\begin{proof}
	First, we compute a list of all small-span $\Z[\omega]$-graphs up to degree $6$.
	This is done by exhaustively growing from $1 \times 1$ small-span matrices $(a)$ where $a \in \Z[\omega]$ is a real element whose absolute value is less than $2.5$.
	We find by inspection that there are no small-span $6 \times 6$ $\Z[\omega]$-matrices that have a row with more than $4$ nonzero entries.
	Now suppose that there exists a small-span $n \times n$ $\Z[\omega]$-matrix with $n \geqslant 7$ and a row having more than $4$ nonzero entries.
	We can take a $6 \times 6$ principal submatrix $B$ that has a row with $5$ nonzero entries.
	By Corollary~\ref{cor:interlaceSpan}, $B$ is small-span.
	But $B$ is not on our computed list.
	Therefore, our supposition must be false.
\end{proof}

It suffices to grow up to size $5 \times 5$ for small-span $R$-matrices where $R \ne \Z[i], \Z[\omega]$ is a quadratic integer ring.

\begin{lemma}\label{lem:boundRowEntries2}
Let $R \ne \Z[\omega]$ be a quadratic integer ring and let $A$ be a small-span Hermitian $R$-matrix with all eigenvalues in the interval $[-2,2.5)$.
Suppose that $A$ is not equivalent to a $\Z$-matrix.
Then each row of $A$ has at most $3$ nonzero entries.
\end{lemma}
\begin{proof}
	Same as the proof of Lemma~\ref{lem:boundRowEntries}, where each quadratic integer ring $R \ne \Z[i]$ is considered separately and the list of $R$-matrices only goes up to size $5 \times 5$.
	Growing from each $2 \times 2$ small-span $R$-matrix that contains an irrational entry ensures that the resulting matrices are not equivalent to $\Z$-matrices.
	
	For the Gaussian integers $\Z[i]$, after growing up to matrices of size $5 \times 5$, one can continue growing from the set of matrices that have a row with $4$ nonzero entries.
	This computation terminates on matrices of size $12 \times 12$ and one observes that all the resulting small-span $\Z[i]$-matrices are equivalent to $\Z$-matrices.
\end{proof}

We can restrict our consideration to $d$ in the set $\left \{-11,-7,-3,-2,-1,2,3,5,6\right \}$.
For other $d$, there are no irrational elements of $\mathcal O_{\Q(\sqrt{d})}$ having (up to Galois conjugation) absolute value small enough to satisfy Lemma~\ref{lem:restrictEntries}.
Hence, for the $d$ not in this set, all small-span $\mathcal O_{\Q(\sqrt{d})}$-matrices are $\Z$-matrices, which have already been classified.

Now, the $1 \times 1$ small-span matrices have the form $(a)$ where $a \in \mathcal O_{\Q(\sqrt{d})}$ is real and $\abs{a} < 2.5$.
Moreover, for $d = 6$ and $n \geqslant 2$, all $n \times n$ small-span matrices are $\Z$-matrices.
Therefore, for $n \times n$ small-span $\mathcal O_{\Q(\sqrt{d})}$-matrices, where $n \geqslant 2$, we can restrict further to $d$ in the set $\{-11,-7,-3,-2,-1,2,3,5\}$.
An indecomposable small-span matrix is called \textbf{maximal} if it is \emph{not} equivalent to any proper submatrix of any indecomposable small-span matrix.
We have computed all $n \times n$ maximal indecomposable small-span $\mathcal O_{\Q(\sqrt{d})}$-matrices (for $2 \leqslant n \leqslant 8$) that are not equivalent to any $\Z$-matrix, and their numbers are tabulated in Table~\ref{tab:numMax} below.
See the author's thesis~\cite{Greaves:Thesis12} for a list of the actual matrices.

\begin{table}[h]
	\begin{center}
	\begin{tabular}{|c|c|c|c|c|c|c|c|c|}
		\hline
		\multirow{2}{*}{$n$} & \multicolumn{8}{c|}{$d$} \\ \cline{2-9}
		 & $-11$ & $-7$ & $-3$ & $-2$ & $-1$ & $2$ & $3$ & $5$ \\
		\hline
		$2$ & $2$ & $2$ & $2$ & $4$ & $2$ & $3$ & $2$ & $2$ \\
		$3$ & $0$ & $6$ & $3$ & $5$ & $8$ & $5$ & $0$ & $4$ \\
		$4$ & $0$ & $7$ & $10$ & $7$ & $16$ & $7$ & $0$ & $10$ \\
		$5$ & $0$ & $8$ & $9$ & $8$ & $10$ & $8$ & $0$ & $1$ \\
		$6$ & $0$ & $4$ & $14$ & $4$ & $6$ & $4$ & $0$ & $0$ \\
		$7$ & $0$ & $2$ & $2$ & $2$ & $3$ & $2$ & $0$ & $0$ \\
		$8$ & $0$ & $2$ & $3$ & $2$ & $3$ & $2$ & $0$ & $0$ \\
		\hline
	\end{tabular}
	\end{center}
	\caption{For each $d$ the number of maximal $n \times n$ small-span $\mathcal O_{\Q(\sqrt{d})}$-matrices that are not equivalent to a $\Z$-matrix.}
	\label{tab:numMax}
\end{table}

Observe from Table~\ref{tab:numMax} that for $d \in \left \{-11,3,5 \right \}$, each small-span $O_{\Q(\sqrt{d})}$-matrix having more than $5$ rows is equivalent to a $\Z$-matrix.
We summarise other useful implications from our computations in the following lemma.

\begin{lemma}\label{lem:SScomps}
	Let $G$ be small-span $O_{\Q(\sqrt{d})}$-graph on more than $6$ vertices.
	Then $G$ has the following properties:
	\begin{enumerate}
		\item Each edge-weight of $G$ has absolute value less than $2$;
		\item Each charge of $G$ has absolute value at most $1$;
		\item No charged vertex is incident to an edge with edge-weight having absolute value more than $1$;
		\item $G$ contains no triangles with fewer than $2$ charged vertices;
		\item If $G$ contains a triangle $T$ having exactly $2$ charged vertices then $T$ is equivalent to the triangle
		
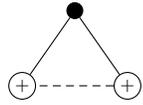
\begin{center}
			\begin{tikzpicture}
			\begin{scope}[scale=1, auto]	
				\foreach \pos/\name/\sign/\charge in {{(0.7,1)/a/zc/{}}, {(0,0)/b/pc/+}, {(1.4,0)/c/pc/+}}
					\node[\sign] (\name) at \pos {$\charge$}; % {$\name$};
				\foreach \edgetype/\source/ \dest /\weight in {pedge/a/b/{}, nedge/c/b/{}, pedge/a/c/{}}
				\path[\edgetype] (\source) -- node[weight] {$\weight$} (\dest);
			\end{scope}
			\end{tikzpicture}.
		\end{center}
		(See the following section for our graph drawing conventions.)
	\end{enumerate}
\end{lemma}

\section{Maximal small-span infinite families}
\label{sec:3}

Let $K$ be a quadratic number field.
We define a \textbf{template} $\mathcal T$ to be a $\C$-graph that is not equivalent to a $\Z$-graph and whose edge-weights are all determined except for some irrational edge-weights $\pm \alpha$ where $\alpha$ is determined only up to its absolute square $a = \alpha \bar \alpha$.
The pair $(\mathcal T, \mathcal O_K)$ is the set of all $\mathcal O_K$-graphs where $\alpha$ is substituted by some element $\rho \in \mathcal O_K$ where $\rho \bar \rho = a$.
We say an $\mathcal O_K$-graph $G$ \textbf{has template} $\mathcal T$ if $G$ is equivalent to some graph in $(\mathcal T, \mathcal O_K)$.
In a template, an edge of weight $1$ is drawn as \tikz[auto] {\path[pedge] (0,0) -- (1.2,0); } and for weight $-1$ we draw \tikz {\path[nedge] (0,0) -- (1.2,0); }.
Let $\alpha_1$ and $\alpha_2$ denote irrational complex numbers with $\abs{\alpha_1}^2 = 1$ and $\abs{\alpha_2}^2 = 2$ respectively.
We draw edges of weight $\alpha_1$ as \tikz[auto] {\path[wedge] (0,0) -- (1.2,0); } and $-\alpha_1$ as \tikz {\path[wnedge] (0,0) -- (1.2,0); }.
Similarly, we draw edges of weight $\pm \alpha_2$ as \tikz[auto] {\path[wwedge] (0,0) -- (1.2,0); } and \tikz {\path[wwnedge] (0,0) -- (1.2,0); }.
For our purposes, we will not need to draw any other types of edges.
A vertex with charge $1$ is drawn as \tikz {\node[pc] {$+$};} and a vertex with charge $-1$ is drawn as \tikz {\node[pc] {$-$};}.
And if a vertex is uncharged, we simply draw \tikz {\node[zc] {};}.
A template only makes sense as a graph over rings that have irrational elements having the required absolute squares. 
Using templates, we can simultaneously study small-span graphs over various quadratic integer rings.

In this section we study the two classes of infinite families of small-span matrices over quadratic integer rings.
Namely, these are the infinite families of templates $\mathcal P_n$ and $\mathcal Q_n$, as drawn below.
(The subscript corresponds to the number of vertices.)
	
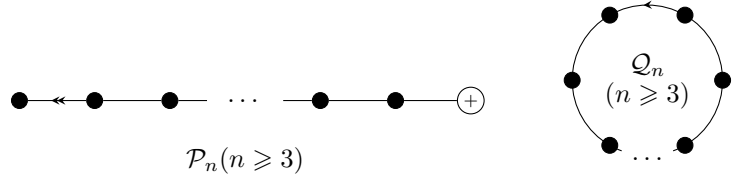
\begin{center}	
	\begin{tikzpicture}
		\begin{scope}
			\foreach \type/\pos/\name in {{vertex/(0,0.5)/bgn}, {vertex/(1,0.5)/b1}, {vertex/(2,0.5)/e1}, {empty/(2.6,0.5)/b11}, {empty/(3.4,0.5)/c11}, {vertex/(4,0.5)/c1}, {vertex/(5,0.5)/d1}}
				\node[\type] (\name) at \pos {};
			\foreach \type/\pos/\name in {{pc/(6,0.5)/f1}}
				\node[\type] (\name) at \pos {$+$};
			\foreach \pos/\name in {{(3,0.5)/\dots}}
				\node at \pos {$\name$};
			\foreach \edgetype/\source/ \dest in {pedge/b1/e1, pedge/e1/b11, pedge/c11/c1, pedge/c1/d1,pedge/d1/f1}
				\path[\edgetype] (\source) -- (\dest);
			\foreach \edgetype/\source/\dest in {wwedge/b1/bgn}
				\path[\edgetype] (\source) -- node[weight] {} (\dest);
			\node at (3,-0.3) {$\mathcal P_n (n \geqslant 3)$};
			\node at (7,-0) {};
		\end{scope}
	\end{tikzpicture}
	\begin{tikzpicture}
		\newdimen\rad
		\rad=1cm
		\newdimen\radi
		\radi=1.04cm
		% \def\shift{0}
	    % Indicate the boundary of the regular polygons
		\draw (271.5:\radi) node[empty] {$\dots$};
		\foreach \y in {60,120,180,240,300}
		{
			\def\x{\y - 120}
			\draw (\x:\rad) node[vertex] {};
			% \draw[pedge] (\x:\radi) -- (\x+225:\radi);
	    }
		\foreach \y in {60,120,240,300}
		{
			\def\x{\y - 120}
			% \draw[pedge] (\x:\radi) -- (\x+225:\radi);
			\draw[pedge] (\x:\rad) arc (\x:\x+60:\rad);
	    }
		\draw[wedge] (60:\rad) arc (60:120:\rad);
		\draw (240:\rad) node[vertex] {};
		\draw[pedge] (240:\rad) arc (240:250:\rad);
		\draw[pedge] (290:\rad) arc (290:300:\rad);
		\node at (0,0.2) {$\mathcal Q_n$};
		\node at (0,-0.2) {$(n \geqslant 3)$};
	\end{tikzpicture}
\end{center}

% \subsection{Classes of infinite families of maximal cyclotomic matrices}

Our classification of small-span $\mathcal O_{\Q(\sqrt{d})}$-matrices uses the classification of cyclotomic $\mathcal O_{\Q(\sqrt{d})}$-matrices.
Since small-span $\Z$-matrices have been classified, we can assume that $d$ is in the set $\left \{-11,-7,-3,-2,-1,2,3,5,6 \right \}$.
We can restrict $d$ further.
In the previous section we observed, from the computations, that on more than $5$ rows there are no small-span $\mathcal O_{\Q(\sqrt{d})}$-matrices when $d$ is $-11$, $3$, or $5$.
Hence, we need only consider $d$ from the set $\left \{-7,-3,-2,-1,2\right \}$.

Cyclotomic $\mathcal O_{\Q(\sqrt{d})}$-graphs have been classified for all $d$, with each classification being over a different set of rings.
In these classifications, the maximal cyclotomic graphs belonging to an infinite family that are not $\Z$-graphs have one of three templates given in Figures~\ref{fig:maxcycs1}, \ref{fig:maxcycs2}, and \ref{fig:maxcycs3}.
We call a template $\mathcal T$ cyclotomic (resp. small-span), if all the elements of $(\mathcal T, \mathcal O_{\Q(\sqrt{d})})$ are cyclotomic (resp. have small-span) for all $d$ that make sense with $\mathcal T$.
(E.g., the family of templates $\mathcal T_{2k}$ in Figure~\ref{fig:maxcycs1} only make sense with $\mathcal O_{\Q(\sqrt{d})}$ when $d=-1$ or $d = -3$.)
More generally, we say a template $\mathcal T$ has property $\mathfrak P$ if all $\mathcal O_{\Q(\sqrt{d})}$-graphs in $(\mathcal T, \mathcal O_{\Q(\sqrt{d})})$ have property $\mathfrak P$ for all appropriate $d$.
Warning: two graphs that have the same template do not necessarily have the same eigenvalues.

In Figures~\ref{fig:maxcycs1}, \ref{fig:maxcycs2}, and \ref{fig:maxcycs3}, we have three infinite families of cyclotomic templates $\mathcal T_{2k}$, $\mathcal C_{2k}$, and $\mathcal C_{2k+1}$.
The sets $(\mathcal T_{2k}, R)$, $(\mathcal C_{2k}, S)$, and $(\mathcal C_{2k+1}, S)$ are sets of maximal cyclotomic graphs where $R = \mathcal O_{\Q(\sqrt{d})}$ for $d = -1$ or $d = -3$, and $S = \mathcal O_{\Q(\sqrt{d})}$ for $d \in \left \{-7,-2,-1,2\right \}$.

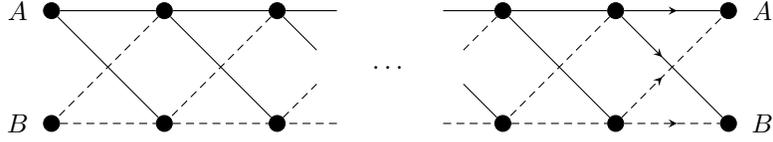
\begin{figure}[htbp]
	\centering
		\begin{tikzpicture}[auto, scale=1.5]
			\begin{scope}
				\foreach \type/\pos/\name in {{vertex/(0,0)/a2}, {vertex/(0,1)/a1}, {vertex/(1,1)/b1}, {vertex/(1,0)/b2}, {vertex/(2,0)/e2}, {vertex/(2,1)/e1}, {empty/(2.6,1)/b11}, {empty/(2.6,0)/b21}, {empty/(2.4,0.6)/b12}, {empty/(2.4,0.4)/b22}, {empty/(3.4,1)/c11}, {empty/(3.4,0)/c21}, {empty/(3.6,0.6)/c12}, {empty/(3.6,0.4)/c22}, {vertex/(4,1)/c1}, {vertex/(4,0)/c2}, {vertex/(5,1)/d1}, {vertex/(5,0)/d2}, {vertex/(6,1)/f1}, {vertex/(6,0)/f2}}
					\node[\type] (\name) at \pos {};
				\foreach \pos/\name in {{(3,0.5)/\dots}, {(-.3,1)/A}, {(-0.3,0)/B}, {(6.3,1)/A}, {(6.3,0)/B}}
					\node at \pos {$\name$};
				\foreach \edgetype/\source/ \dest/\weight in {nedge/b1/a2/{}, pedge/a1/b1/{}, pedge/a1/b2/{}, nedge/a2/b2/{}, nedge/e1/b2/{}, pedge/b1/e1/{}, pedge/b1/e2/{}, nedge/b2/e2/{}, nedge/b21/e2/{}, pedge/e1/b11/{}, pedge/e1/b12/{}, nedge/e2/b22/{}, pedge/c11/c1/{}, nedge/c12/c1/{}, pedge/c22/c2/{}, nedge/c21/c2/{}, nedge/d1/c2/{}, pedge/c1/d1/{}, pedge/c1/d2/{}, nedge/c2/d2/{}, wedge/d1/f1/{}, wnedge/d2/f2/{}}
					\path[\edgetype] (\source) -- node[weight] {$\weight$} (\dest);
				% \node at (3,-0.4) {$T^{(x)}_{2k} (k \geqslant 3)$};
			\end{scope}
			\begin{scope}[decoration={markings,mark = at position 0.4 with {\arrow{stealth} } }]
				\foreach \edgetype/\source/ \dest/\weight in {dnedge/d2/f1/{}, dpedge/d1/f2/{}}
					\path[\edgetype] (\source) -- node[weight] {$\weight$} (\dest);
			\end{scope}
		\end{tikzpicture}
	\caption{The infinite family $\mathcal T_{2k}$ of $2k$-vertex maximal connected cyclotomic templates. (The two copies of vertices $A$ and $B$ should each be identified to give a toral tessellation.) }
	\label{fig:maxcycs1}
\end{figure}

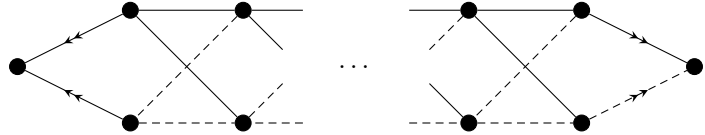
\begin{figure}[htbp]
	\centering
		\begin{tikzpicture}[scale=1.5, auto]
			\begin{scope}
				\foreach \type/\pos/\name in {{vertex/(1,1)/b1}, {vertex/(1,0)/b2}, {vertex/(2,0)/e2}, {vertex/(2,1)/e1}, {empty/(2.6,1)/b11}, {empty/(2.6,0)/b21}, {empty/(2.4,0.6)/b12}, {empty/(2.4,0.4)/b22}, {empty/(3.4,1)/c11}, {empty/(3.4,0)/c21}, {empty/(3.6,0.6)/c12}, {empty/(3.6,0.4)/c22}, {vertex/(4,1)/c1}, {vertex/(4,0)/c2}, {vertex/(5,1)/d1}, {vertex/(5,0)/d2}}
					\node[\type] (\name) at \pos {};
				\foreach \type/\pos/\name in {{vertex/(0,0.5)/bgn}, {vertex/(6,0.5)/end}}
					\node[\type] (\name) at \pos {};
				\foreach \pos/\name in {{(3,0.5)/\dots}}
					\node at \pos {\name};
				\foreach \edgetype/\source/ \dest in {nedge/e1/b2, pedge/b1/e1, pedge/b1/e2, nedge/b2/e2, nedge/b21/e2, pedge/e1/b11, pedge/e1/b12, nedge/e2/b22, pedge/c11/c1, nedge/c12/c1, pedge/c22/c2, nedge/c21/c2, nedge/d1/c2, pedge/c1/d1, pedge/c1/d2, nedge/c2/d2}
					\path[\edgetype] (\source) -- (\dest);
				\foreach \edgetype/\source/\dest in {wwedge/b1/bgn, wwedge/b2/bgn, wwnedge/d2/end, wwedge/d1/end}
					\path[\edgetype] (\source) -- (\dest);
				% \node at (3,-0.4) {$C_{2k} (k \geqslant 2)$};
			\end{scope}
		\end{tikzpicture}
	\caption{The infinite family of $2k$-vertex maximal connected cyclotomic templates $\mathcal C_{2k}$ for $k \geqslant 2$.}
	\label{fig:maxcycs2}
\end{figure}

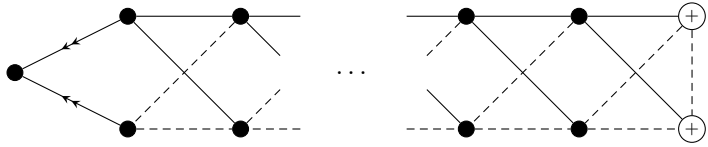
\begin{figure}[htbp]
	\centering
		\begin{tikzpicture}[scale=1.5, auto]
			\begin{scope}
				\foreach \type/\pos/\name in {{vertex/(0,0.5)/bgn}, {vertex/(1,1)/b1}, {vertex/(1,0)/b2}, {vertex/(2,0)/e2}, {vertex/(2,1)/e1}, {empty/(2.6,1)/b11}, {empty/(2.6,0)/b21}, {empty/(2.4,0.6)/b12}, {empty/(2.4,0.4)/b22}, {empty/(3.4,1)/c11}, {empty/(3.4,0)/c21}, {empty/(3.6,0.6)/c12}, {empty/(3.6,0.4)/c22}, {vertex/(4,1)/c1}, {vertex/(4,0)/c2}, {vertex/(5,1)/d1}, {vertex/(5,0)/d2}}
					\node[\type] (\name) at \pos {};
				\foreach \type/\pos/\name in {{pc/(6,1)/f1}, {pc/(6,0)/f2}}
					\node[\type] (\name) at \pos {$+$};
				\foreach \pos/\name in {{(3,0.5)/\dots}}
					\node at \pos {$\name$};
				\foreach \edgetype/\source/ \dest in {nedge/e1/b2, pedge/b1/e1, pedge/b1/e2, nedge/b2/e2, nedge/b21/e2, pedge/e1/b11, pedge/e1/b12, nedge/e2/b22, pedge/c11/c1, nedge/c12/c1, pedge/c22/c2, nedge/c21/c2, nedge/d1/c2, pedge/c1/d1, pedge/c1/d2, nedge/c2/d2, nedge/f1/d2, pedge/d1/f1, pedge/d1/f2, nedge/d2/f2, nedge/f1/f2}
					\path[\edgetype] (\source) -- (\dest);
				\foreach \edgetype/\source/\dest in {wwedge/b1/bgn, wwedge/b2/bgn}
					\path[\edgetype] (\source) -- (\dest);
				% \node at (3,-0.4) {$C_{2k+1} (k \geqslant 1)$};
			\end{scope}
		\end{tikzpicture}
	\caption{The infinite family of $(2k+1)$-vertex maximal connected cyclotomic templates $\mathcal C_{2k+1}$ for $k \geqslant 1$.}
	\label{fig:maxcycs3}
\end{figure}

All of the eigenvalues of the maximal connected cyclotomic graphs, in the sets $(\mathcal T_{2k}, R)$, $(\mathcal C_{2k}, S)$, and $(\mathcal C_{2k+1}, S)$, are equal to $\pm 2$ with at least one pair of eigenvalues having opposite signs.
Hence each of these graphs has span equal to $4$.
We want to find small-span graphs contained inside these maximal connected cyclotomic templates.
We look for subgraphs that have span equal to $4$.
This way, we can recognise that a graph is not small-span if it contains one of these as a subgraph.

Since its spectrum is $\{-2^{(1)},0^{(2)},2^{(1)}\}$, we have that the $\Z$-graph $X_4$ in Figure~\ref{fig:x4} has span equal to $4$.

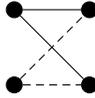
\begin{figure}[htbp]
	\centering
		\begin{tikzpicture}
			\begin{scope}[yshift=-4.2cm, xshift=1.5cm]
				\foreach \type/\pos/\name in {{vertex/(1,1)/b1}, {vertex/(1,0)/b2}, {vertex/(2,0)/e2}, {vertex/(2,1)/e1}}
					\node[\type] (\name) at \pos {};
				\foreach \edgetype/\source/ \dest in {pedge/b1/e1, pedge/e2/b1, nedge/b2/e1, nedge/b2/e2}
					\path[\edgetype] (\source) -- (\dest);
				% \node at (1.6,-0.6) {$X_4$};
			\end{scope}
		\end{tikzpicture}
	\caption{The $\Z$-graph $X_4$.}
	\label{fig:x4}
\end{figure}

In Figure~\ref{fig:span4graphs} we list $4$ infinite families of cyclotomic sub-templates having span equal to $4$.

\begin{figure}[htbp]
	\centering
	\begin{tikzpicture}
	\begin{scope}[yshift=2cm]
		\foreach \type/\pos/\name in {{vertex/(0,0.5)/bgn}, {vertex/(1,0.5)/b1}, {vertex/(2,0.5)/e1}, {empty/(2.6,0.5)/b11}, {empty/(3.4,0.5)/c11}, {vertex/(4,0.5)/c1}, {vertex/(5,0.5)/d1}, {vertex/(6,0.5)/end}}
			\node[\type] (\name) at \pos {};
		\foreach \pos/\name in {{(0,0.1)/\bar \alpha_2}, {(1,0.1)/2}, {(2,0.1)/2}, {(4,0.1)/2}, {(5,0.1)/2}, {(6,0.1)/\bar \alpha_2}}
			\node at \pos {\small $\name$};
		\foreach \pos/\name in {{(3,0.5)/\dots}}
			\node at \pos {$\name$};
		\foreach \edgetype/\source/ \dest in {pedge/b1/e1, pedge/e1/b11, pedge/c11/c1, pedge/c1/d1}
			\path[\edgetype] (\source) -- (\dest);
		\foreach \edgetype/\source/\dest in {wwedge/b1/bgn, wwedge/d1/end}
			\path[\edgetype] (\source) -- node[weight] {} (\dest);
		\node at (3,-0.6) {$\mathcal X^{(1)}_n (n \geqslant 3)$};
	\end{scope}
	\end{tikzpicture}
	
		\begin{tikzpicture}[scale=1, auto]
			\begin{scope}[yshift=0.1cm]
				\foreach \type/\pos/\name in {{vertex/(0,0.5)/bgn}, {vertex/(1,0.5)/b1}, {vertex/(2,0.5)/e1}, {empty/(2.6,0.5)/b11}, {empty/(3.4,0.5)/c11}, {vertex/(4,0.5)/c1}, {vertex/(5,0.5)/d1}}
					\node[\type] (\name) at \pos {};
				\foreach \type/\pos/\name in {{pc/(6,1)/f1}, {pc/(6,0)/f2}}
					\node[\type] (\name) at \pos {$+$};
				\foreach \pos/\name in {{(0,0.1)/e_n}, {(1,0.1)/e_{n-1}}, {(2,0.1)/e_{n-2}}, {(4,0.1)/e_4}, {(5,0.1)/e_3}, {(6.5,-0)/e_2}, {(6.5,1)/e_1}}
					\node at \pos {\small $\name$};
				\foreach \pos/\name in {{(3,0.5)/\dots}}
					\node at \pos {$\name$};
				\foreach \edgetype/\source/ \dest in {pedge/b1/e1, pedge/e1/b11, pedge/c11/c1, pedge/c1/d1,pedge/d1/f1, pedge/d1/f2, nedge/f1/f2}
					\path[\edgetype] (\source) -- (\dest);
				\foreach \edgetype/\source/\dest in {wwedge/b1/bgn}
					\path[\edgetype] (\source) -- node[weight] {} (\dest);
				\node at (3,-0.6) {$\mathcal X^{(2)}_n (n \geqslant 4)$};
			\end{scope}
			\end{tikzpicture}
			\begin{tikzpicture}
			\begin{scope}[yshift=-2cm,xshift=0.3cm]
				\foreach \type/\pos/\name in {{vertex/(0,0.5)/bgn}, {vertex/(1,0.5)/b1}, {vertex/(2,0.5)/e1}, {empty/(2.6,0.5)/b11}, {empty/(3.4,0.5)/c11}, {vertex/(4,0.5)/c1}, {vertex/(5,0.5)/d1}, {vertex/(5.7,1)/f1}, {vertex/(5.7,0)/f2}}
					\node[\type] (\name) at \pos {};
				\foreach \pos/\name in {{(0,0.1)/\bar \alpha_2}, {(1,0.1)/2}, {(2,0.1)/2}, {(4,0.1)/2}, {(5,0.1)/2}, {(5.7,-0.3)/1}, {(5.7,1.3)/1}}
					\node at \pos {$\name$};
				\foreach \pos/\name in {{(3,0.5)/\dots}}
					\node at \pos {$\name$};
				\foreach \edgetype/\source/ \dest in {pedge/b1/e1, pedge/e1/b11, pedge/c11/c1, pedge/c1/d1, pedge/d1/f1, pedge/d1/f2}
					\path[\edgetype] (\source) -- (\dest);
				\foreach \edgetype/\source/\dest in {wwedge/b1/bgn}
					\path[\edgetype] (\source) -- node[weight] {} (\dest);
				\node at (3,-0.6) {$\mathcal X^{(3)}_n (n \geqslant 4)$};
			\end{scope}
			\end{tikzpicture}
			\begin{tikzpicture}
			\begin{scope}[yshift=-3.8cm]
				\foreach \type/\pos/\name in {{vertex/(0,0.5)/bgn}, {vertex/(1,0.5)/b1}, {vertex/(2,0.5)/e1}, {empty/(2.6,0.5)/b11}, {empty/(3.4,0.5)/c11}, {vertex/(4,0.5)/c1}, {vertex/(5,0.5)/d1}, {vertex/(5.7,1)/f1}, {vertex/(5.7,0)/f2}, {vertex/(6.4,0.5)/g1}, {vertex/(7.4,0.5)/h1}, {empty/(8,0.5)/i11}, {empty/(8.8,0.5)/j11}, {vertex/(9.4,0.5)/k1}, {vertex/(10.4,0.5)/l1}, {vertex/(11.4,0.5)/end}}
					\node[\type] (\name) at \pos {};
				\foreach \pos/\name in {{(-0.5,0.7)/A}, {(0,0.1)/l_s}, {(1,0.1)/l_{s-1}}, {(2,0.1)/l_{s-2}}, {(4,0.1)/l_{2}}, {(5,0.1)/l_1}, {(5.7,-0.3)/r_{0}}, {(5.7,1.3)/l_{0}}, {(6.4,0.1)/r_1}, {(7.4,0.1)/r_{2}}, {(9.4,0.1)/r_{t-2}}, {(10.4,0.1)/r_{t-1}}, {(11.4,0.1)/r_t = l_s}, {(11.9,0.7)/A}}
					\node at \pos {$\name$};
				\foreach \pos/\name in {{(3,0.5)/\dots}, {(8.4,0.5)/\dots}}
					\node at \pos {$\name$};
				\foreach \edgetype/\source/ \dest in {pedge/b1/e1, pedge/e1/b11, pedge/c11/c1, pedge/c1/d1, pedge/d1/f1, pedge/d1/f2, nedge/g1/f2, pedge/f1/g1, pedge/g1/h1, pedge/h1/i11, pedge/j11/k1, pedge/k1/l1, pedge/l1/end}
					\path[\edgetype] (\source) -- (\dest);
				\foreach \edgetype/\source/\dest in {wedge/b1/bgn}
					\path[\edgetype] (\source) -- node[weight] {} (\dest);
				\node at (5.7,-0.8) {$\mathcal X^{(4)}_{s,t} (s,t \geqslant 2)$};
			\end{scope}
		\end{tikzpicture}
	\caption{Cyclotomic templates having span equal to $4$. In the first three templates, the subscript denotes the number of vertices. The last template has $s+t+1$ vertices and the two copies of the vertex $A$ should be identified.}
	\label{fig:span4graphs}
\end{figure}
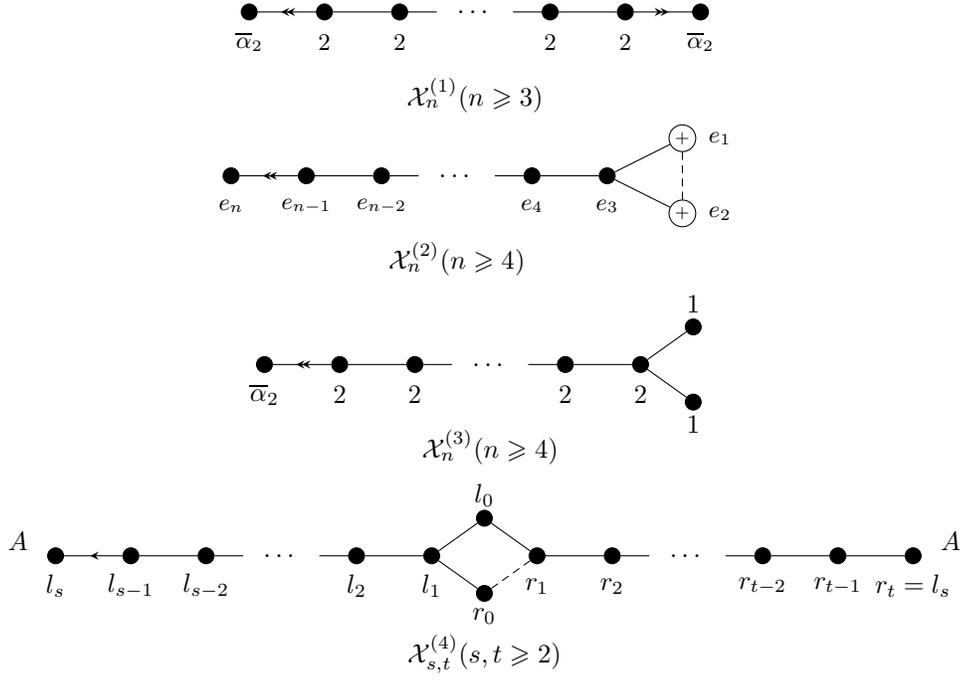

We can show that the sub-templates $\mathcal X^{(1)}_n$, $\mathcal X^{(2)}_n$, $\mathcal X^{(3)}_n$, and $\mathcal X^{(4)}_{s,t}$ in Figure~\ref{fig:span4graphs} have span equal to $4$ by simply giving the eigenvectors corresponding to the eigenvalues $\pm 2$.
% For each $\mathcal O_{\Q(\sqrt{d})}$-graph in $(\mathcal X^{(1)}_n, \mathcal O_{\Q(\sqrt{d})})$ or $(\mathcal X^{(3)}_n, \mathcal O_{\Q(\sqrt{d})})$, we use $\alpha_2$ to label the edge-weights of absolute square $2$.
The numbers beneath the vertices of $\mathcal X^{(1)}_n$ and $\mathcal X^{(3)}_n$ in Figure~\ref{fig:span4graphs} correspond to entries of an eigenvector having associated eigenvalue $2$ and since these two templates are bipartite, they also have as an eigenvalue $-2$.
The entries of the eigenvectors of $\mathcal X^{(2)}_n$ are represented by the $e_j$; the eigenvector associated to $2$ has $e_1 = 1$, $e_2 = -1$, and $e_j = 0$ for $j \in \left \{3,\dots, n\right \}$.
For the eigenvalue $-2$ we have $e_1 = e_2 = 1$, $e_j = (-1)^j \cdot 2$ for $j \in \left \{3,\dots, n-1\right \}$, and $e_n = (-1)^n \cdot \bar \alpha_2$.
For the graph $\mathcal X^{(4)}_{s,t}$, we are either working over the Gaussian integers $\Z[i]$ or the Eisenstein integers $\Z[\omega]$.
For the eigenvalue
\begin{itemize}
	\item[$-2$:] 
	for $j \in \left \{1,\dots,s-1\right \}$, set
	\[
		l_j = (-1)^{s-j}(1+\alpha_1), \quad  r_j = (-1)^{t-j}(1 + \bar\alpha_1), \quad \text{ and } \quad l_s = r_t = 1 + \bar \alpha_1.
	\]
	If $s+t$ is even, set
	\[
		l_0 = (-1)^s\left ( 1 + \frac{\alpha_1 + \bar \alpha_1}{2} \right ) \quad \text{ and } \quad r_0 = (-1)^s \frac{\alpha_1 - \bar \alpha_1}{2}.
	\]
	Otherwise set
	\[
		l_0 = (-1)^s \frac{\alpha_1 - \bar \alpha_1}{2} \quad \text{ and } \quad r_0 = (-1)^s\left ( 1 + \frac{\alpha_1 + \bar \alpha_1}{2} \right ).
	\]
	
	\item[$2$:] 
	set
	\[
		l_1 = \dots = l_{s-1} = 1+\alpha_1, \; r_1 = \dots = r_{t} = l_s = 1 + \bar \alpha_1,
	\]
	and
	\[
		l_0 = 1 + \frac{\alpha_1 + \bar \alpha_1}{2} \text{ and } r_0 = \frac{\alpha_1 - \bar \alpha_1}{2}.
	\]
\end{itemize}

In the next two lemmata we show that any small-span sub-template of $\mathcal C_{2k}$, $\mathcal C_{2k+1}$, or $\mathcal T_{2k}$ is a sub-template of either $\mathcal P_n$ or $\mathcal Q_n$.

\begin{lemma}\label{lem:cycPn}
	The template $\mathcal P_n$ is small-span for all $n \geqslant 3$.
	Any connected small-span sub-template of either $\mathcal C_{2k}$ or $\mathcal C_{2k+1}$ is contained in $\mathcal P_n$ for some $n$.
\end{lemma}

Recall that a template must have at least one irrational edge-weight.

\begin{proof} 
	The template $\mathcal P_n$ is a sub-template of the cyclotomic template $\mathcal C_{2n-1}$ and so, by interlacing, the eigenvalues of $\mathcal P_n$ lie in the interval $[-2,2]$. 	
	Let $A$ be an adjacency matrix of $\mathcal P_n$. 
	Clearly, it suffices to show that $\mathcal P_n$ does not have $-2$ as an eigenvalue, i.e., the rows of $A + 2I$ are linearly independent.
	Using Gaussian elimination, put $A+2I$ into the following upper triangular form.
	\[
		\begin{pmatrix}
			3 & 1 & & \\
			 & 5/3 & 1 & \\
			& & 7/5 & 1 \\
			& & & \ddots & \ddots & \\
			& & & & \frac{2(n - 2) + 1}{2(n-2) - 1} & 1 \\
			& & & & & \frac{2(n - 1) + 1}{2(n-1) - 1} & \alpha_2 \\
			& & & & & & 2\big(1 - \frac{2(n - 1) - 1}{2(n-1) + 1}\big) \\
		\end{pmatrix}
	\]
	It is then easy to see that the determinant is $4$. 
	Therefore $A+2I$ is non-singular and so $A$ does not have $-2$ as an eigenvalue.

	By Corollary~\ref{cor:interlaceSpan}, any subgraph of a graph having either $\mathcal X^{(1)}_k$, $\mathcal X^{(2)}_k$, or $\mathcal X^{(3)}_k$ as a template cannot occur as a subgraph of a small-span graph.
	Similarly, $X_4$ cannot be a subgraph of a small-span graph.
	Hence all connected small-span sub-templates of $\mathcal C_{2k}$ and $\mathcal C_{2k+1}$ are sub-templates of $\mathcal P_n$ for some $n$.
\end{proof}

\begin{lemma}\label{lem:cycQn}
	The template $\mathcal Q_n$ is small-span for all $n \geqslant 3$.
	Any connected small-span sub-template of $\mathcal T_{2k}$ is contained in $\mathcal Q_n$ for some $n$.
\end{lemma}

\begin{proof}
	As with the previous lemma, since $\mathcal Q_n$ is contained in $\mathcal T_{2n}$, it suffices to show that $\mathcal Q_n$ does not have $-2$ as an eigenvalue.
	Let $A$ be an adjacency matrix of $\mathcal Q_n$.
	Using Gaussian elimination, put $A+2I$ into the following upper triangular form.
	\[
		\begin{pmatrix}
			2 & 1 &  &  & &  & 1 \\
			 & 3/2 & 1 & & & & -1/2 \\
			 &  & 4/3 & 1 & & & 1/3  \\
			 & &  & \ddots & \ddots & &  \vdots \\
			& & &  & \frac{n-1}{n-2} & 1 & \frac{(-1)^{n-1}}{n-2} \\
			& & & &  & \frac{n}{n-1} & \alpha_1 + \frac{(-1)^n}{n-1} \\
			& & &  & &  & 1 + S(n) \\
		\end{pmatrix}.
	\]
	Here $$S(n) = \frac{1}{n}\big (1 - (-1)^n(\alpha_1+\bar\alpha_1) \big ) - \sum_{k=2}^n\frac{1}{k(k-1)}.$$ 
	Hence the determinant of this matrix is
	\[
		2 - (-1)^n(\alpha_1+\bar\alpha_1).
	\]
	Since the absolute value of $\alpha_1$ is $1$ and $\alpha_1$ is not an integer, the absolute value of the real part of $\alpha_1$ is less than $1$, and so the determinant is nonzero for all $n$.
	Therefore $\mathcal Q_n$ is small-span.
	
	By Corollary~\ref{cor:interlaceSpan}, no small-span template can contain $\mathcal X^{(4)}_{t,s}$ (see Figure~\ref{fig:span4graphs}) for all $s,t \geqslant 2$.
	Moreover, any subgraph of $\mathcal T_{2k}$ obtained by deleting two vertices that have the same neighbourhood is a $\Z$-graph and hence not a template.
	With these two restrictions on the subgraphs of $\mathcal T_{2k}$, we are done. 
\end{proof}

In the next proposition we classify cyclotomic small-span $\mathcal O_{\Q(\sqrt{d})}$-matrices, using the classification of cyclotomic matrices over quadratic integer rings.

\begin{proposition}\label{pro:cycSS}
	Let $\mathcal T$ be a connected cyclotomic small-span template on more than $6$ vertices.
	Then $\mathcal T$ is contained in either $\mathcal P_n$ or $\mathcal Q_n$ for some $n$.
\end{proposition}

\begin{proof}
	We can readily check the subgraphs of sporadic cyclotomic graphs of over quadratic integer rings (see the classifications \cite{GTay:cyclos10,Greaves:CycloEG11,Greaves:CycloRQ11}) that are not equivalent to a $\Z$-graph, to find that no such subgraph on more than $6$ vertices is small-span.
	The theorem then follows from Lemma~\ref{lem:cycPn} and Lemma~\ref{lem:cycQn}.
\end{proof}

Before completing the classification of small-span $\mathcal O_{\Q(\sqrt{d})}$-matrices, we first state some lemmata.

\begin{lemma}\label{lem:stayConnected}\cite[Lemma 2]{McKee:SmallSpan10}
	Let $G$ be a connected graph and let $u$ and $v$ be vertices of $G$ such that the distance from $u$ to $v$ is maximal.
	Then the subgraph induced by removing $u$ (and its incident edges) is connected.
\end{lemma}

\begin{lemma}\label{lem:ISMCycSmallSpan12}\cite[Theorem 3]{McKee:SmallSpan10}
	Let $A$ be a small-span $\Z$-graph on more than $12$ vertices.
	Then $A$ is cyclotomic.
\end{lemma}

\begin{lemma}\label{lem:simpatheq}\cite[Theorem 2.2]{Seid:Signed94}
	Let $G$ be a path whose edge-weights $\alpha$ satisfy $\abs{\alpha} = 1$.
	Then $G$ is strongly equivalent to a $\{0,1\}$-path.
\end{lemma}

\begin{lemma}\label{lem:IQcycleequiv}
	Let $G$ be an uncharged cycle all of whose edge-weights have absolute value $1$.
	Then $G$ is strongly equivalent to a graph having $\mathcal Q_n$ as a template for some $n \in \N$.
\end{lemma}
\begin{proof}
	Suppose $G$ has $n$ vertices.
	Label the vertices $v_1, \ldots, v_n$ so that $v_1$ is adjacent to $v_n$ and $v_j$ is adjacent to $v_{j+1}$ for all $j \in \left \{1,\ldots,n-1\right \}$.
	We can inductively switch the vertices of $G$ so that $w(v_j,v_{j+1}) = 1$ for all $j \in \left \{1,\ldots,n-1\right \}$, and $\abs{w(v_1,v_n)} = 1$.
\end{proof}

We are now ready to prove the following proposition.

\begin{proposition}\label{pro:SSmbCyc}
	Let $\mathcal T$ be a small-span template on more than $8$ vertices.
	Then $\mathcal T$ is cyclotomic.
\end{proposition}

% We show that for $d \in \left \{-7,-2,2\right \}$, any connected small-span $\mathcal O_{\Q(\sqrt{d})}$-graph $G$ on more than $6$ vertices is contained in a graph having as its template $\mathcal P_n$ for some $n$;

\begin{proof}
	Suppose that $\mathcal T^\prime$ is a counterexample on the minimal number of vertices possible. 
	Then $\mathcal T^\prime$ is non-cyclotomic and has at least $9$ vertices.
	It follows from the minimality of $\mathcal T^\prime$ that any proper sub-template of $\mathcal T^\prime$ must have all its eigenvalues in the interval $[-2,2]$.
	But if $\mathcal T^\prime$ does not have any proper sub-templates, i.e., if all its subgraphs are equivalent to $\Z$-graphs, then $\mathcal T^\prime$ must be a cycle and its subpaths need not have all their eigenvalues in $[-2,2]$.
	Pick vertices $u$ and $v$ as far apart as possible in $\mathcal T^\prime$. 
	By Lemma~\ref{lem:stayConnected}, deleting either $u$ or $v$ leaves a connected subgraph on at least $8$ vertices.
	If the subgraph $G_u$ obtained by deleting $u$ from $\mathcal T^\prime$ is equivalent to a $\Z$-graph then let $G = G_v$ be the subgraph obtained by deleting $v$; otherwise let $G = G_u$.
	If $G$ is a template, then by Proposition~\ref{pro:cycSS}, either $G$ is equivalent to $\mathcal Q_n$ for some $n$ or $G$ is equivalent to a connected subgraph of $\mathcal P_n$ for some $n$.
	Otherwise, if $G$ is equivalent to a $\Z$-graph, then $\mathcal T^\prime$ must be a cycle and each irrational edge-weight $\alpha$ must satisfy $\abs{\alpha} = 1$.
	Moreover, by Lemma~\ref{lem:SScomps}, we have that every edge-weight of $\mathcal T^\prime$ must have absolute value equal to $1$.
	First we deal with the case where $G$ is equivalent to a $\Z$-graph.
	
	\paragraph{\emph{Case 1.}} Suppose $G$ is equivalent to a $\Z$-graph.
	Then, by above, $\mathcal T^\prime$ is a cycle whose edge-weights all have absolute value $1$.
	Let $C$ be a cycle on $n \geqslant 9$ vertices $v_0,v_1,\dots, v_{n-1}$ such that $v_j$ is adjacent to $v_{j+1}$ with subscripts reduced modulo $n$.
	Let each vertex $v_j$ of $C$ have charge $c_j \in \left \{-1,0,1\right \}$ and let each edge-weight have absolute value $1$.
	Now, if $C$ is small-span then each induced subpath of $C$ is small-span.
	Hence, $C$ is small-span only if each subpath $v_k v_{k+1} \cdots v_{k+n-2}$ is small-span for all $k$ where the subscripts are reduced modulo $n$.
	Since we are working up to equivalence, by Lemma~\ref{lem:simpatheq} we can assume that all edges in these paths have weight $1$.
	We have computed all small-span $\Z[i]$ and $\Z[\omega]$-cycles on up to $13$ vertices having all edge-weight of absolute value $1$.
	The ones on more than $8$ vertices have $\mathcal Q_n$ (for some $n$) as a template.
	By Lemma~\ref{lem:ISMCycSmallSpan12} all small-span $\Z$-paths on at least $13$ vertices are cyclotomic and hence have charges (if at all) only at each end-vertex.
	It follows that $C$ is small-span only if $c_j = 0$ for all $j \in \left \{0,\dots, n-1\right \}$.
	Therefore, $\mathcal T^\prime$ must be uncharged.
	Hence $\mathcal T^\prime$ is equivalent to $\mathcal Q_n$ for some $n$ (see Lemma~\ref{lem:IQcycleequiv}).
	This contradicts $\mathcal T^\prime$ being non-cyclotomic. 
	Note that if $C$ had only $8$ vertices then the above argument would not work since the following $\Z[\omega]$-cycle $\mathfrak C_8$ is small-span:
	
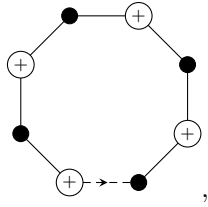
\[
	\begin{tikzpicture}
		\newdimen\rad
		\rad=1.2cm

	    % Indicate the boundary of the regular polygons
		\draw[pedge] (157.5:\rad) -- (202.5:\rad);
		\draw[wnedge] (247.5:\rad) -- (292.5:\rad);
		\draw[pedge] (337.5:\rad) -- (382.5:\rad);
		\draw[pedge] (67.5:\rad) -- (112.5:\rad);
		\foreach \x in {112.5,202.5,292.5,382.5}
		{
			\draw[pedge] (\x:\rad) -- (\x+45:\rad);			% 
						% \draw[pedge] (\x+45:\rad) -- (\x+90:\rad);
			\draw (\x:\rad) node[vertex] {};
			\draw (\x+45:\rad) node[pc] {$+$};
	    }
	\end{tikzpicture},
	\]
	where the irrational edge-weight is $-\omega$.
	
	Now we have two remaining cases to consider: the case where $G$ is contained in $\mathcal P_n$ and the case where $G$ is equivalent to $\mathcal Q_n$ for some $n$.
	In each case, by Lemma~\ref{lem:SScomps}, we can exclude the possibility that $G$ contains a triangle having fewer than two charged vertices.
	And similarly, by Lemma~\ref{lem:SScomps}, we need only consider charges from the set $\left \{-1,0,1\right \}$.
	Moreover, we can exclude any graph equivalent to $X^{(6)}_5$, $X^{(7)}_5$, $X^{(8)}_5$, or $X^{(9)}_8$ (see Figure~\ref{fig:x6}) as a subgraph of a small-span graph since these graphs have span at least $4$.
	\begin{figure}[htbp]
		\centering
			\begin{tikzpicture}
				\begin{scope}[yshift=-4.2cm, xshift=1.5cm]
					\foreach \type/\pos/\name/\charge in {{pc/(0,0)/b1/{+}}, {vertex/(1,0)/b2/{}}, {vertex/(2,0)/e2/{}}, {vertex/(3,0)/e3/{}}, {vertex/(1,0.8)/e1/{}}}
						\node[\type] (\name) at \pos {$\charge$};
					\foreach \edgetype/\source/ \dest in {pedge/b1/b2, pedge/e2/b2, pedge/b2/e1, pedge/e3/e2}
						\path[\edgetype] (\source) -- (\dest);
					\node at (1.6,-0.6) {$X^{(6)}_5$};
				\end{scope}
			\end{tikzpicture}
			\begin{tikzpicture}
				\begin{scope}[yshift=-4.2cm, xshift=1.5cm]
					\foreach \type/\pos/\name/\charge in {{pc/(0,0)/b1/{+}}, {vertex/(1,0)/b2/{}}, {vertex/(2,0)/e2/{}}, {vertex/(3,0)/e3/{}}, {pc/(1,0.8)/e1/{+}}}
						\node[\type] (\name) at \pos {$\charge$};
					\foreach \edgetype/\source/ \dest in {pedge/b1/b2, pedge/e2/b2, pedge/b2/e1, pedge/e3/e2}
						\path[\edgetype] (\source) -- (\dest);
					\node at (1.6,-0.6) {$X^{(7)}_5$};
				\end{scope}
			\end{tikzpicture}
			\begin{tikzpicture}
				\begin{scope}[yshift=-4.2cm, xshift=1.5cm]
					\foreach \type/\pos/\name/\charge in {{pc/(0,0)/b1/{+}}, {vertex/(1,0)/b2/{}}, {vertex/(2,0)/e2/{}}, {vertex/(3,0)/e3/{}}, {pc/(1,0.8)/e1/{-}}}
						\node[\type] (\name) at \pos {$\charge$};
					\foreach \edgetype/\source/ \dest in {pedge/b1/b2, pedge/e2/b2, pedge/b2/e1, pedge/e3/e2}
						\path[\edgetype] (\source) -- (\dest);
					 \node at (1.6,-0.6) {$X^{(8)}_5$};
				\end{scope}
			\end{tikzpicture}
			
			\begin{tikzpicture}
				\begin{scope}[yshift=-4.2cm, xshift=1.5cm]
					\foreach \type/\pos/\name/\charge in {{vertex/(-2,0)/z1/{}}, {vertex/(-1,0)/a1/{}}, {vertex/(0,0)/b1/{}}, {vertex/(1,0)/b2/{}}, {vertex/(2,0)/e2/{}}, {vertex/(3,0)/e3/{}}, {vertex/(4,0)/e4/{}}, {vertex/(1,0.8)/e1/{}}}
						\node[\type] (\name) at \pos {$\charge$};
					\foreach \edgetype/\source/ \dest in {pedge/z1/a1, pedge/a1/b1, pedge/b1/b2, pedge/e2/b2, pedge/b2/e1, pedge/e3/e2, pedge/e3/e4}
						\path[\edgetype] (\source) -- (\dest);
					\node at (1,-0.6) {$X^{(9)}_8$};
				\end{scope}
			\end{tikzpicture}
			
			\begin{tikzpicture}
				\begin{scope}[yshift=-4.2cm, xshift=1.5cm]
					\foreach \type/\pos/\name/\charge in {{vertex/(-2,0)/z1/{}}, {pc/(-1,0)/a1/{+}}, {vertex/(0,0)/b1/{}}, {vertex/(1,0)/b2/{}}, {vertex/(2,0)/e2/{}}, {vertex/(3,0)/e3/{}}, {vertex/(4,0)/e4/{}}, {vertex/(5,0)/e5/{}}, {vertex/(6,0)/e6/{}}}
						\node[\type] (\name) at \pos {$\charge$};
					\foreach \edgetype/\source/ \dest in {pedge/z1/a1, pedge/a1/b1, pedge/b1/b2, pedge/e2/b2, pedge/e3/e2, pedge/e3/e4, pedge/e5/e4, pedge/e5/e6}
						\path[\edgetype] (\source) -- (\dest);
					\node at (2,-0.6) {$X^{(10)}_9$};
				\end{scope}
			\end{tikzpicture}
			
			\begin{tikzpicture}
				\begin{scope}[yshift=-4.2cm, xshift=1.5cm]
					\foreach \type/\pos/\name/\charge in {{vertex/(-2,0)/z1/{}}, {pc/(-1,0)/a1/{+}}, {vertex/(0,0)/b1/{}}, {vertex/(1,0)/b2/{}}, {vertex/(2,0)/e2/{}}, {vertex/(3,0)/e3/{}}, {vertex/(4,0)/e4/{}}, {vertex/(5,0)/e5/{}}, {vertex/(6,0)/e6/{}}}
						\node[\type] (\name) at \pos {$\charge$};
					\foreach \edgetype/\source/ \dest in {pedge/z1/a1, pedge/a1/b1, pedge/b1/b2, pedge/e2/b2, pedge/e3/e2, pedge/e3/e4, pedge/e5/e4, wwedge/e5/e6}
						\path[\edgetype] (\source) -- (\dest);
					\node at (2,-0.6) {$X^{(11)}_9$};
				\end{scope}
			\end{tikzpicture}
			
			\begin{tikzpicture}
				\begin{scope}[yshift=-4.2cm, xshift=1.5cm]
					\foreach \type/\pos/\name/\charge in {{vertex/(-2,0)/z1/{}}, {vertex/(-1,0)/a1/{}}, {vertex/(0,0)/b1/{}}, {vertex/(1,0)/b2/{}}, {vertex/(2,0)/e2/{}}}
						\node[\type] (\name) at \pos {$\charge$};
					\foreach \edgetype/\source/ \dest in {pedge/z1/a1, wwedge/a1/b1, pedge/b1/b2, pedge/e2/b2}
						\path[\edgetype] (\source) -- (\dest);
					\node at (0,-0.6) {$X^{(12)}_5$};
				\end{scope}
			\end{tikzpicture}
			\begin{tikzpicture}
				\begin{scope}[yshift=-4.2cm, xshift=1.5cm]
					\foreach \type/\pos/\name/\charge in {{pc/(-2,0)/z1/{+}}, {vertex/(-1,0)/a1/{}}, {vertex/(0,0)/b1/{}}, {vertex/(1,0)/b2/{}}, {vertex/(2,0)/e2/{}}}
						\node[\type] (\name) at \pos {$\charge$};
					\foreach \edgetype/\source/ \dest in {pedge/z1/a1, wwedge/a1/b1, pedge/b1/b2, pedge/e2/b2}
						\path[\edgetype] (\source) -- (\dest);
					\node at (0,-0.6) {$X^{(13)}_5$};
				\end{scope}
			\end{tikzpicture}
		\caption{The graphs $X^{(6)}_5$, $X^{(7)}_5$, $X^{(8)}_5$, $X^{(9)}_8$, $X^{(10)}_9$, $X^{(11)}_9$, $X^{(12)}_5$, and $X^{(13)}_5$ which all have span greater than or equal to $4$.}
		\label{fig:x6}
	\end{figure}
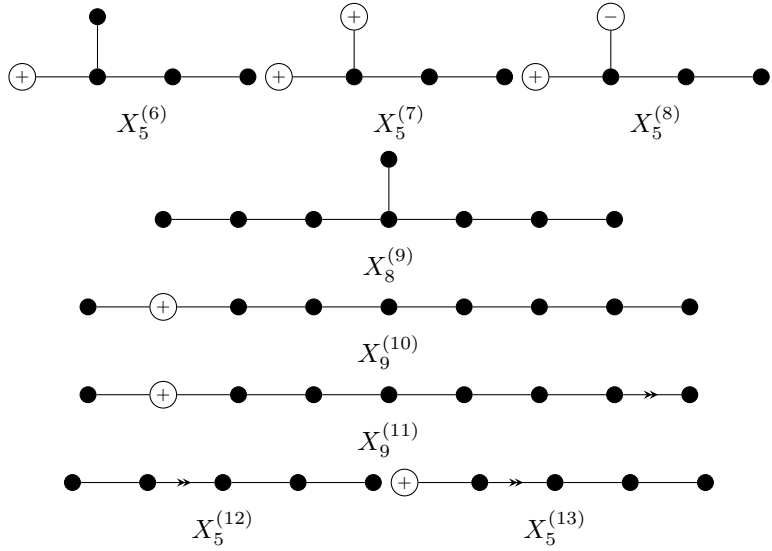
	
	\paragraph{\emph{Case 2.}} Suppose $G$ is equivalent to a sub-template of $\mathcal P_n$ for some $n$.
	In this case $\mathcal T^\prime$ cannot be a cycle for, if it were, it would have a sub-template on more than $6$ vertices that is not equivalent to a subgraph of either $\mathcal P_n$ or $\mathcal Q_n$ contradicting Proposition~\ref{pro:cycSS}.
	By Lemma~\ref{lem:SScomps}, a charged vertex cannot be incident to an edge-weight of absolute value greater than $1$.
	Now, $G$ has at most one charged vertex.
	Hence, if $\mathcal T^\prime$ contains a triangle then it can have at most two charged vertices.
	We deduce therefore that $\mathcal T^\prime$ must be triangle-free, since otherwise, by Lemma~\ref{lem:SScomps}, $\mathcal T^\prime$ would contain $\mathcal X^{(2)}_n$ for some $n$.
	Moreover, since we have excluded $X^{(6)}_5$, $X^{(7)}_5$, and $X^{(8)}_5$, no leaf (a vertex having only one neighbour) can share its adjacent vertex with a charged leaf.
	
	Since it is equivalent to a subgraph of $\mathcal P_n$, we have that $G$ is a path.
	Recall that we obtained $G$ by deleting either $u$ or $v$ which are vertices at the maximal distance from one another.
	Label the vertices of $G$ by $v_1, v_2, \dots, v_{r}$ where $v_j$ is adjacent to $v_{j+1}$ for $j \in \{1,\dots,r-1\}$.
	Then we can obtain $\mathcal T^\prime$ from $G$ by attaching a vertex to one of the vertices $v_1$, $v_2$, $v_{r-1}$, or $v_{r}$.
	Thus, in order for $\mathcal T^\prime$ not to contain a subgraph equivalent to $X^{(6)}_5$, $X^{(7)}_5$, $X^{(8)}_5$, $X^{(10)}_9$, $X^{(11)}_9$, $X^{(12)}_5$, $X^{(13)}_5$, $\mathcal X^{(1)}_k$, or $\mathcal X^{(3)}_k$, the template $\mathcal T^\prime$ must be equivalent to a sub-template of either $\mathcal X^{(1)}_k$, $\mathcal X^{(3)}_k$, or $\mathcal P_k$ for some $k$.
	Since each of these templates are cyclotomic, we have established a contradiction.
	
	\paragraph{\emph{Case 3.}} Suppose $G$ is equivalent to $\mathcal Q_n$ for some $n$.
	Since we have excluded graphs equivalent to $X^{(6)}_5$ and $X^{(9)}_8$ as subgraphs, and triangles with fewer than two charged vertices are forbidden, there do not exist any possible small-span supergraphs $\mathcal T^\prime$ having $G$ as a subgraph.
\end{proof}

\section{Missing small-span polynomials}
\label{sec:4}

Finally, we turn our attention to the question of which polynomials appear as minimal polynomials of small-span matrices.
We define a \textbf{cosine polynomial} to be a monic integer polynomial having all its zeros contained in the interval $[-2,2]$, and a \textbf{non-cosine polynomial} to be a monic totally-real integer polynomial with at least one zero lying outside $[-2,2]$.
McKee \cite{McKee:SmallSpan10} found six small-span polynomials of low degree that are not the minimal polynomial of any symmetric $\Z$-matrix: three degree-$6$ cosine polynomials
\begin{align*}
	& x^6 - x^5 - 6x^4 + 6x^3 + 8x^2 - 8x + 1, \\
	& x^6 - 7x^4 + 14x^2 - 7, \\
	& x^6 - 6x^4 + 9x^2 - 3,
\end{align*}
and three degree-$7$ non-cosine polynomials
\begin{align*}
	& x^7 - x^6 - 7x^5 + 5x^4 + 15x^3 -5x^2- 10x - 1, \\
	& x^7 - 8x^5 + 19x^3 - 12x-1, \\
	& x^7 - 2x^6 - 6x^5 + 11x^4 + 11x^3 -17x^2- 6x +7.
\end{align*}
Suppose a non-cosine polynomial of degree more than $8$ is the minimal polynomial of a Hermitian $R$-matrix for some quadratic integer ring $R$.
Then, by Proposition~\ref{pro:SSmbCyc}, it is the minimal polynomial of an integer symmetric matrix.
In fact, by our computations, except for the graph $\mathfrak C_8$ (mentioned in the proof of Proposition~\ref{pro:SSmbCyc}), the above holds for non-cosine polynomials of degree more than $6$.
We record this observation as a proposition, below. 
Let $Q(x)$ be the characteristic polynomial of $\mathfrak C_8$.
We have established the following proposition.

\begin{proposition}\label{pro:newStuff}
	Let $R$ be a quadratic integer ring and let $f \ne \pm Q(x)$ be a non-cosine polynomial of degree more than $6$.
	If $f$ is the minimal polynomial of some Hermitian $R$-matrix then $f$ is the minimal polynomial of some integer symmetric matrix.
\end{proposition}

Hence, in particular, the three degree-$7$ non-cosine polynomials are not minimal polynomials of any Hermitian $R$-matrix for any quadratic integer ring $R$.

We do, however, find that two of the degree-$6$ cosine polynomials above are characteristic polynomials of Hermitian $\Z[\omega]$-matrices.
The graph
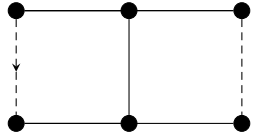
\begin{center}
	\begin{tikzpicture}[scale=1.5, auto]
		\foreach \pos/\name/\sign/\charge in {{(0,0)/a/zc/{}}, {(1,0)/b/zc/{}}, {(0,1)/c/zc/{}}, {(1,1)/d/zc/{}}, {(2,0)/e/zc/{}}, {(2,1)/f/zc/{}}}
			\node[\sign] (\name) at \pos {$\charge$};
		\foreach \edgetype/\source/\dest/\weight in {{pedge/a/b/{}}, {pedge/b/d/{}}, {pedge/d/c/{}}, {wnedge/c/a/{}}, {pedge/c/f/{}}, {pedge/a/e/{}}, {nedge/f/e/{}}, {pedge/f/c/{}}}
			\path[\edgetype] (\source) -- node[weight] {$\weight$} (\dest);
	\end{tikzpicture}
\end{center}
where $\alpha_1 = \omega$, has characteristic polynomial $x^6 - 7x^4 + 14x^2 - 7$; and the graphs $\mathcal Q_6$ and
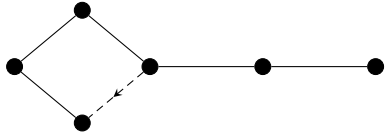
\begin{center}
	\begin{tikzpicture}[scale=1.5, auto]
		\foreach \pos/\name/\sign/\charge in {{(-0.6,0.5)/a/zc/{}}, {(0,1)/b/zc/{}}, {(0,0)/c/zc/{}}, {(0.6,0.5)/d/zc/{}}, {(1.6,0.5)/e/zc/{}}, {(2.6,0.5)/f/zc/{}}}
			\node[\sign] (\name) at \pos {$\charge$};
		\foreach \edgetype/\source/\dest/\weight in {{pedge/a/b/{}}, {pedge/b/d/{}}, {wnedge/d/c/{}}, {pedge/c/a/{}}, {pedge/d/e/{}}, {pedge/f/e/{}}}
			\path[\edgetype] (\source) -- node[weight] {$\weight$} (\dest);
	\end{tikzpicture}
\end{center}
where $\alpha_1 = \omega$ in both, have characteristic polynomial $x^6 - 6x^4 + 9x^2 - 3$.

The polynomial $p(x) = x^6 - x^5 - 6x^4 + 6x^3 + 8x^2 - 8x + 1$ remains somewhat elusive as it is \emph{not} the minimal polynomial of any Hermitian $R$-matrix for any quadratic integer ring $R$.
For suppose that $p(x)$ is the minimal polynomial of some Hermitian $R$-matrix $A$.
Then each eigenvalue of $A$ is a zero of $p(x)$ \cite[\S 11.6]{Hartley:1970fk} and the minimal polynomial $p(x)$ divides the characteristic polynomial $\chi_A(x)$.
Since $p(x)$ is irreducible, $\chi_A(x)$ must be some power of $p(x)$.
Therefore, we need to check that $p(x)$ is not the minimal polynomial of any $r \times r$ Hermitian $R$-matrix where $6$ divides $r$.
Both $\mathcal P_n$ and $\mathcal Q_n$ have span larger than the span of $p(x)$ for $n \geqslant 18$ and we have checked all possible matrices for $r = 6$ and $12$.
Hence we have Theorem~\ref{thm:main}.

Finally, we point out a matrix that comes close to having $p(x)$ as its minimal polynomial: the $\Z[\omega]$-graph $\mathcal Q_7$ where $\alpha_1 = -\omega$, has as its characteristic polynomial the polynomial $(x+1)(x^6 - x^5 - 6x^4 + 6x^3 + 8x^2 - 8x + 1)$.

\section{Acknowledgement}

The author is grateful for the improvements to the presentation suggested by the referee.

\bibliographystyle{plain} 
\bibliography{/Users/Gary/Dropbox/Papers/bib}
\end{document}